\documentclass[11pt,bezier]{article}
\usepackage{amsmath,amssymb,amsfonts,euscript,graphicx}
\usepackage[table,xcdraw]{xcolor}
\usepackage{pgf, tikz}
\usetikzlibrary{arrows, automata}

\usepackage{graphicx}
\usepackage{pgfplots,multicol}

\usepackage{mathtools}
\usepackage{amsfonts}
\usepackage[english]{babel}

\usepackage[all]{xy}
\usepackage{mathtools}
\usepackage{tikz-cd}

\newtheorem{preproof}{{\bf \indent Proof.}}

\newenvironment{proof}[1]{\begin{preproof}{\rm
               #1}\hfill{$\Box$}}{\end{preproof}}

\newtheorem{prop}{\bf\indent Proposition}[section]
\newtheorem{defn}[prop]{\bf\indent Definition}
\newtheorem{cor}[prop]{\bf\indent Corollary}
\newtheorem{example}[prop]{\bf\indent Example}
\newtheorem{thm}[prop]{\bf\indent Theorem}

\title{\bf  \large On pseudo-absorbing primary multiplication modules\\ over pullback rings}
\author{{\normalsize {\sc M.J. Nikmehr${}^{\mathsf{a}}$}\thanks{Corresponding author} , {\sc R. Nikandish${}^{\mathsf{b}}$ and {\sc A. Yassine${}^{\mathsf{a}}$}  }
}\vspace{3mm}\\
{\footnotesize{}}\\
{\footnotesize{}}\\
{\footnotesize{${}^{\mathsf{a}}$\it Faculty of Mathematics, K.N. Toosi
University of Technology, }}\\
{\footnotesize{\rm P.O. BOX \rm{16315-1618}, Tehran, Iran}}\\
{\footnotesize{ $\mathsf{nikmehr@kntu.ac.ir}$}}\quad\quad
{\footnotesize{$\mathsf{yassine\_ali@email.kntu.ac.ir}$}}\\
{\footnotesize{${}^{\mathsf{b}}$\it Department of Mathematics, Jundi-Shapur University of Technology,}}\\
{\footnotesize{\rm P.O. BOX \rm{64615-334},
Dezful, Iran}}\\
{\footnotesize{ $\mathsf{r.nikandish@ipm.ir}$}}\\
{\footnotesize{$\mathsf{}$ }}}
\date{}

\begin{document}

\maketitle
\begin{abstract}
{A famous result due to L. S. Levy provides a classification of all finitely generated indecomposable modules over Dedekind-like rings. This motivates us to outline an approach to the classification of indecomposable pseudo-absorbing primary multiplication modules with finite-dimensional top over certain kinds of pullback rings. In this paper, we give a complete classification, up to isomorphism, of all indecomposable pseudo-absorbing primary multiplication modules with finite-dimensional top over a pullback of two valuation domains with the same residue field. We also find a connection between pseudo-absorbing primary multiplication modules and pure-injective modules over such domains.}
\end{abstract}
\begin{center}{\section{Introduction
}}\end{center}
\par
A basic problem in mathematics is to classify arbitrary modules one is studying up to isomorphism. A lesson many authors learned from representation types (finite, tame, or wild) is that while this is almost always impossible (see \cite{1,28,29} and \cite[Chapter 19]{30}), it is sometimes possible, and very useful, to classify certain full subcategories of the category one is working in. In particular, the pure-injective modules and also indecomposable representable modules have been classified \cite{Ebrahimipure,Ebrahimi}. Indeed, the classification of the complete theories of modules reduces to classifying the pure-injectives. Notably, there are important classifications of modules like the classification of finite-dimensional and finitely generated modules, in many cases this classification can be extended to classify pure-injective modules. In fact, there is a connection between infinitely generated pure-injective modules and finitely generated modules in many cases (see \cite{5,15,18}). The purpose of this paper is to introduce a class of modules called pseudo-absorbing primary multiplication modules which form an interesting subclass of pure-injective modules.

Modules over pullback rings have been studied by several authors (see \cite{Bass,Dolati,Ebrahimipure,Ebrahimisemiprime,EsmaeiliSaraei,Kirichenko,Klingler} and \cite{Wiseman}). The first important contribution in this direction is due to Levy \cite{Levy2}, who classified all finitely generated indecomposable modules over Dedekind-like rings. Later, this classification was extended by Klingler \cite{Klingler} to lattices over certain non-commutative Dedekind-like rings, and then Haefner and Klingler classified lattices over certain non-commutative pullback rings, which they called special quasi-triads \cite{Haefner1,Haefner2}. The common of all these classifications is the reduction to a ``matrix problem"  over a division ring \cite{Drozd}.

We assume throughout this paper that all rings are commutative with identity. Recall the construction: if $R_1$ and $R_2$ are two  discrete valuation domains with maximal ideals $P_1 = Rp_1$ and $P_2 = Rp_2$ respectively, then 
$R = (R_1 \xrightarrow{\mathcal{V}_1(x)} \overline{R} \xleftarrow{\mathcal{V}_2(x)} R_2)$
 is the pullback ring of $R_1$ and $R_2$ with a unique maximal ideal $P = P_1\oplus P_2$. Note that  $\overline{R} \cong R/P \cong R_1/P_1\cong R_2/P_2$  is a field. This means that $R$ is a commutative Noetherian local ring with unique maximal ideal $P$. So it is easy to see that $P_1$ and $P_2$ are the other prime ideals of $R$. The elements of this ring are of the form $a = (r, s)$ where $r\in R_1$ and $s\in R_2$ such that $\mathcal{V}_1(r) = \mathcal{V}_2(s)$. If $r$ and $s$ are nonzero, then $a = (p_1^n, p_2^m)$ for some positive integers $m,n$, which means that ann$(a) = 0$. Therefore $Ra \cong R$. But if $a = (0, p_2^m)$ ($a = (p_1^n, 0)$) for some positive integers $m,n$, then ann$(a) =  P_1\oplus 0$ (ann$(a) =0\oplus P_2$), and thus $R/(P_1 \oplus 0)\cong R(0, p_2^m)\cong R_2$ ($R/(0 \oplus P_2)\cong R(p_1^n, 0)\cong R_1$). The other ideals $I$ of $R$ are of the form $I = P_1^n \oplus P_2^m = (P^n_1, P^m_2) = (<p^n_1>, <p^m_2>)$ for some positive integers $m,n$.

Following some ideas and a technique introduced in S. Ebrahimi Atani et al. \textit {Pseudo-absorbing multiplication modules over a pullback ring} (J. Pure Appl. Algebra 222  (2018), 3124--3136), in this paper we introduce a class of modules called pseudo-absorbing primary multiplication modules which form a subclass of pure-injective modules (see Definition \ref{*}), and contains the class of pseudo-absorbing multiplication modules \cite{Dolati} and the class of pseudo-prime multiplication modules \cite{EsmaeiliSaraei}. We are mainly interested to study these modules in the two following cases: where $R$ is a Dedekind domain or a pullback ring of two local Dedekind domains. In section 2, a complete description of pseudo-absorbing primary multiplication modules over a discrete valuation domain and some of their basic properties are given. The classification of these modules is divided into two steps: indecomposable separated pseudo-absorbing primary multiplication and indecomposable non-separated pseudo-absorbing primary multiplication $R$-modules. In section 3, the indecomposable separated pseudo-absorbing primary multiplication modules over a pullback of two local Dedekind domains with finite-dimensional top over $R/Rad(R)$ are completely described. In section 4, we use the list of separated pseudo-absorbing primary multiplication modules obtained in section 3 and results of Levy \cite{Levy2,Levy3} on the possibilities of amalgamating finitely generated separated modules to show that non-separated indecomposable pseudo-absorbing primary multiplication $R$-modules with finite-dimensional top are precisely  finite direct sums of separated indecomposable separated pseudo-absorbing primary multiplication $R$-modules. We will see also that non-separated modules may be represented by certain amalgamation chains of indecomposable separated pseudo-absorbing primary multiplication modules and this abutment corresponds to amalgamation in the socles of these separated pseudo-absorbing primary multiplication modules.

For completeness, we define the concepts that will be needed. A proper ideal $I$ of a commutative ring $R$ is called \textit{$2$-absorbing primary} if whenever $a, b, c \in R$ and $abc \in I$, then $ab \in I$ or $ac \in \sqrt{I}$ or $bc \in \sqrt{I}$ \cite{Badawi2}. This notion was extended to the concept of \textit{$2$-absorbing primary submodule} in \cite{Mostafanasab} which is a generalization of prime submodule with the following definition, if whenever $a, b \in R$ and $m \in M$ and $abm \in N$ for some proper submodule $N$ of $M$, then $am \in N$ or $bm \in N$ or $ab \in \sqrt{(N : M)}$. An $R$-module $M$ is called \textit{multiplication}  if for each submodule $N$ of $M$, $N = IM$, for some ideal $I$ of $R$ \cite{Bast}. A submodule $N$ of an $R$-module $M$ is said to be \textit{pure}  if any finite system of equations over $N$ which is solvable in $M$ is also solvable in $N$. A submodule $N$ of an $R$-module $M$ is said to be \textit{relatively divisible} in $M$ if $rN = N \cap rM$ for all $r \in R$ \cite{Prest}. A module $M$ is said to be \textit{pure-injective} if it has the injective property relative to all pure exact sequences \cite{Ebrahimipure,Prest}. An $R$-module $S$ is said to be \textit{separated} if there exist $R_i$-modules $S_i$, $i = 1,2$, such that $S$ is a submodule of $S_1 \oplus S_2$. Let $R$ be a pullback ring as above. Equivalently, $S$ is \textit{separated} if and only if $P_1S \cap P_2S = 0$ \cite{Levy}. Let $M$ be an $R$-module. A \textit{separated representation} of $M$ is an epimorphism $\varphi: S \rightarrow M$ of $R$-modules where $S$ is separated and $\varphi$ is an $R$-homomorphism, and if $\varphi$ admits a factorization $\varphi = (S \xrightarrow{f} S' \rightarrow M)$ such that $S'$ is separated, then $f$ is one-to-one. The module $K = $ Ker$(\varphi )$ is an $\overline{R}$-module, since $\overline{R} = R/P$ and $PK = 0$ \cite[Proposition 2.3]{Levy}. An exact sequence $0 \rightarrow K \xrightarrow{i} S \xrightarrow{\varphi} M \rightarrow 0$ of $R$-modules with $S$ separated and $K$ an $\overline{R}$-module is a separated representation of $M$ if and only if $P_iS \cap K = 0$ for each $i$ and $K \subseteq PS$ \cite[Proposition 2.3]{Levy3}. It follows from \cite[Theorem 2.8]{Levy} that every module $M$ has a separated representation, which is unique up to isomorphism. For undefined notation and terminology, the reader is referred to \cite{Ebrahimipure,Ebrahimi,Dolati}.

\vspace*{1cm}
\begin{center}{\section{Some properties of pseudo-absorbing primary multiplication modules
}}\end{center}
In this section, the notion of pseudo-absorbing primary multiplication modules over a discrete valuation domain is introduced and
some of their basic properties are given.

\begin{defn} \label{*}
Let $R$ be a ring and $M$ be an $R$-module. A proper submodule $N$ of $M$ is said to be pseudo-absorbing primary if $(N :_R M)$ is a $2$-absorbing primary ideal of $R$. Furthermore, $M$ is said to be a pseudo-absorbing primary multiplication module if for every pseudo-absorbing primary submodule $N$ of $M$, there is an ideal $I$ of $R$ such that $N = IM$.
\end{defn}

It is easy to see that if $N$ is a pseudo-absorbing primary submodule of a pseudo-absorbing primary multiplication module $M$, then there exists an ideal $I$ of $R$ such that $N = IM$, and so  $N = IM \subseteq (N :_R M)M\subseteq N$. Hence $N = (N :_R M)M$ for every pseudo-absorbing primary submodule $N$ of $M$. Since every prime ideal is a $2$-absorbing ideal and every $2$-absorbing ideal is a $2$-absorbing primary ideal, one can easily see that the class of pseudo-absorbing primary multiplication modules contains the class of pseudo-absorbing multiplication modules and the class of pseudo-absorbing multiplication modules contains the class of pseudo-prime multiplication modules. 

The following example shows that a pseudo-absorbing primary submodule of $M$ need not be a primary (prime) or pseudo-absorbing or $2$-absorbing primary submodule of $M$.

\begin{example} {\rm 
$(1)$	 Let $M = \mathbb{Z}_{16} = \{0, 1, 2, \ldots , 15\}$ be an $R$-module such that $R = \mathbb{Z}$ and $N = \{0, 8\}$ a submodule of $M$. Then $(N :_R M) = 8\mathbb{Z}$ and $\sqrt{(N :_R M)} = 2\mathbb{Z}$. It is easy to see that $N$ is a $2$-absorbing primary submodule of $M$, and so it is a pseudo-absorbing primary submodule of $M$. Now, $2\cdot 2\cdot 2 \in (N :_R M)$ but $2\cdot 2 \notin (N :_R M)$. Hence, $N$ is a pseudo-absorbing primary submodule of $M$ which is not pseudo-absorbing.

$(2)$ Let $N = 6\mathbb{Z}$ be a submodule of a $\mathbb{Z}$-module $\mathbb{Z}$. Since $N$ is a $2$-absorbing primary submodule of $M$, $N$ is a pseudo-absorbing primary submodule of $M$. But it is not a primary (prime) submodule, as $2\cdot 3 \in N$, whereas $2 \notin N$ and $3 \notin \sqrt{(N :_R M)}$.

$(3)$ It is easy to see that if $N$ is a $2$-absorbing primary submodule of an $R$-module $M$, then $N$ is a pseudo-absorbing primary submodule of $M$. Asuume that $R$ is a  discrete valuation domain with  maximal deal $P = Rp$. It follows from \cite[Lemma 2.6]{Ebrahimi} that every non-zero proper submodule $L$ of $E = E(R/P)$, the injective hull of $R/P$, is of the form $L = A_n = (0 :_E P^n)$ $(n \geq 1)$, $L = A_n = Ra_n$ and $PA_{n+1} = A_n$. One can see that no $A_n$ is a 2-absorbing primary submodule of $E$,  for otherwise, if $n$ is a positive integer, then $P^2A_{n+2} = A_n$, and neither $PA_{n+2} = A_{n+1} \subseteq A_n$ nor $P^2 \subseteq \sqrt{(A_n : E)} = 0$. Hence no $A_n$ is a 2-absorbing primary submodule of $E$. Now, $(A_n : E) = 0$ implies $A_n$ is a pseudo-absorbing primary submodule for all $n \geq 1$, and thus a pseudo-absorbing primary submodule of $M$ need not be a $2$-absorbing primary submodule of $M$.

}
\end{example}

\begin{prop} \label{2.3}
Suppose that $R$ is a  ring and $N \subseteq K$ are submodules of an $R$-module $M$. Then the
following statements hold:

$(1)$  If $K$ is a pseudo-absorbing primary submodule of $M$, then $K/N$ is apseudo-absorbing primary submodule of $M/N$.

$(2)$  If $I$ is an ideal of $R$ such that $I \subseteq (0 :_R M)$ and $K$ is a pseudo-absorbing primary submodule of $M$, then $K$ is a pseudo-absorbing primary submodule of $M$ as an $R/I$-module.

$(3)$ If $K$ is a $2$-absorbing primary and pseudo-prime submodule of $M$, then $K$ is a $2$-absorbing submodule of $M$. The converse is true whenever prime ideals of $R$ are comparable; in particular, $R$ is local with maximal
ideal $M$ and for every minimal prime ideal $P$ over a $2$-absorbing ideal $I$, we have $IM = P$.
\end{prop}
\begin{proof}
{$(1)$ This is obvious because $(K :_R M) = (K/N :_R M/N)$.

$(2)$ Let $(K :_R M) = Q$ for some $2$-absorbing primary ideal $Q$ of $R$. Since $IM = 0$, $I \subseteq Q$, and hence an inspection will show that $(K :_{R/I} M) = Q/I$. Thus $K$ is a pseudo-absorbing primary submodule of $M$ as an $R/I$-module.

$(3)$	 Suppose that $K$ is a $2$-absorbing submodule of $M$. It follows from \cite[Theorem 2.3]{Bennis} that $K$ is a pseudo-prime submodule of $M$. Conversely, suppose that $K$ is a $2$-absorbing primary and pseudo-prime submodule of $M$, and $abm \in K$ for some $a,b \in R$ and $m \in M$. Suppose also that $am, bm \notin K$. Then $(ab)^n \in (K :_R M)$ for some positive integer $n$, since $K$ is a $2$-absorbing primary submodule of $M$. But $K$ is a pseudo-prime submodule of $M$, hence $ab \in (K :_R M)$. Thus $K$ is a $2$-absorbing submodule of $M$.
}
\end{proof}

\begin{prop} \label{242}
Suppose that $f: R \to R'$ is a homomorphism and $M$ is a pseudo-absorbing primary multiplication $R$-module and suppose also that $N$ is a submodule of $M$. Then the following statements hold:

$(1)$ For every ideal $I$ of $R$ such that $I \subseteq (N :_R M)$ and $N$ is a non-zero of $M$, $M/N$ is a pseudo-absorbing primary multiplication $R/I$-module. In particular, $M/N$ is a pseudo-absorbing primary multiplication $R$-module.

$(2)$ Every direct summand of $M$ is a pseudo-absorbing primary multiplication submodule of $M$.

$(3)$ If $I \subseteq (0 :_R M)$ where $I$ is an ideal of $R$, then $M$ is a pseudo-absorbing primary multiplication $R$-module if and
only if $M$ is pseudo-absorbing primary multiplication as an $R/I$-module.

$(4)$ If $M'$ is a pseudo-absorbing primary multiplication $R'$-module, then $M'$ is a pseudo-absorbing primary multiplication $R$-module.
\end{prop}
\begin{proof}
{$(1)$ Suppose that $K/N$ is a pseudo-absorbing primary submodule of $M/N$. It follows from Proposition \ref{2.3} $(1)$, $K$ is a pseudo-absorbing primary submodule of $M$, and hence $K = (K : M)M$. This means that $I \subseteq (N :_R M) \subseteq (K :_R M) = J$. An inspection shows that $K/N = (J/I)(M/N)$. For the``in particular" statement take $I = 0$. 

$(2)$ This follows from part $(1)$.

$(3)$	 It follows from Proposition \ref{2.3} $(2)$ that $N$ is a pseudo-absorbing primary submodule of $M$ if and only if $N$ is a pseudo-absorbing primary $R/I$-submodule of $M$. Now the assertion follows from the fact that $(N :_R M) = (N :_{R/I} M)$.

$(4)$ Suppose that $f : R \to R'$ is a surjective homomorphism and $M'$ is a pseudo-absorbing primary multiplication $R'$-module. Suppose also that $N$ is a pseudo-absorbing primary $R$-submodule of $M'$. Then $N = IM'$ for some ideal $I$ of $R'$. Since $f$ is a surjective homomorphism, there is and ideal $J$ of $R$ such that $J = f^{-1}(I)$ with $f(J) = I$. Then $JM' = f(J)M' = IM' = N$. Hence $M'$ is a pseudo-absorbing primary multiplication $R$-module.
}
\end{proof}

In the following result, we show that if $R$ is a  discrete valuation domain with  maximal deal $P$, then the injective hull of $R/P$ and the field of fractions of $R$ are not pseudo-absorbing primary submodules.

\begin{prop} \label{2.5.5}
 Suppose that $R$ is a  discrete valuation domain with  maximal deal $P$. Then the injective hull $E(R/P)$ of $R/P$ and the field of fraction $Q(R)$ of $R$ are not pseudo-absorbing primary multiplication modules. Furthermore, the cyclic modules $R$ and $R/P^n$ are pseudo-absorbing primary multiplication modules.
\end{prop}
\begin{proof}
{First we show that every divisible pseudo-absorbing primary multiplication module $M$ over an integral domain $R$ is a simple module. Suppose that $N$ is a proper submodule of $M$. Since $M$ is divisible, $(N :_R M) = 0$, and so $N = (N :_R M)M = 0_M = 0$. Thus $M$ is a simple module. Clearly $E(R/P)$ and $Q(R)$ are divisible modules and by \cite[Lemma 2.6]{Ebrahimi}, $E(R/P)$ has nonzero proper submodules and the submodule $N = \{r/1 : r \in R\}$ is a nonzero proper submodule of $Q(R)$. Therefore $E(R/P)$ and $Q(R)$ are not simple, and hence the injective hull $E(R/P)$ of $R/P$ and the field of fraction $Q(R)$ of $R$ are not pseudo-absorbing primary  multiplication modules. Now, since the cyclic modules $R$ and $R/P^n$ are multiplication, they are also pseudo-absorbing primary multiplication modules.
}
\end{proof}

The following result contains a complete list of indecomposable pseudo-absorbing primary  multiplication modules over local Dedekind domains.

\begin{thm} \label{prev}
Suppose that $R$ is a  discrete valuation domain with  maximal ideal $P = Rp$. Then the indecomposable pseudo-absorbing primary multiplication modules over $R$, up to isomorphism, are:

$(1)$ $R$;

$(2)$ $R/P^n$, $n \geq 1$ the indecomposable torsion modules.

\end{thm}
\begin{proof}
{It follows from \cite[Proposition 1.3]{Ebrahimipure} and Proposition \ref{2.5.5} that the modules $R$ and $R/P^n$ ($n \geq 1$) are indecomposable pseudo-absorbing primary multiplication. It remains to show that there are no more indecomposable pseudo-absorbing primary multiplication $R$-modules. For otherwise we assume that $M$ is an indecomposable pseudo-absorbing primary multiplication module, and suppose that $0\neq a \in M$. Suppose also that $h(a) =$
sup$\{n : a \in P^nM\}$ (this means that $h(a)$ is either a nonnegative integer or $\infty$). One can see that either $(0 :_R a) = P^m$ or $(0 :_R a) = 0$, since if $(0 :_R a) = P^{m+1}$, $p^ma \neq 0$ and $pp^ma = 0$, and so we can take $a$ such that $(0 :_R a) = P$ or $0$. We consider three cases.
\newline
{\bf Case one:} If $h(a) = n$ and $(0 :_R a) = P$, then $a = p^nb$ for some positive integer $n$ and $b \in M$. Therefore $Rb \cong R/P^{n+1}$ is a pseudo-absorbing primary multiplication $R$-module. Hence by a similar argument like that in \cite[Theorem 2.12 Case 1]{Atani}, the submodule $Rb$ of $M$ is pure. But $Rb$ is a submodule of bounded order of $M$, hence \cite[Theorem 5]{Kaplansky} implies that $Rb$ is a direct summand of $M$. Thus $M = Rb \cong R/P^{n+1}$.
\newline
{\bf Case two:} If $h(a) = n$ and $(0 :_R a) = 0$, then $a = p^nb$ for some positive integer $n$ and $b \in M$. Therefore $(0 :_R b) = 0$, and hence $Rb \cong R$. By a similar argument like that in Case 1, the submodule $Rb$ of $M$ is pure. Hence by a similar argument like that in \cite[Theorem 2.12 Case 2]{Atani}, there is a positive integer $t$ such that $R \cong Rb = P^tM = M$.
\newline
{\bf Case three:} Assume that $h(a) = \infty$. If $(0 :_R a) = P$, then as in \cite[Theorem 2.12 Case 4]{Atani} we have $M \cong E(R/P)$, a contradiction, by Proposition \ref{2.5.5}. If $(0 :_R a) = 0$, then as in \cite[Theorem 2.12 Case 3]{Atani} we have $M \cong Q(R)$, a contradiction, by Proposition \ref{2.5.5}.
}
\end{proof}

We close this section with the following result.

\begin{cor} 
Suppose that $R$ is a  discrete valuation domain with  maximal ideal $P = Rp$, and $M$ is a pseudo-absorbing primary multiplication $R$-module. Then $M$ is a direct sum of copies of $R/P^n$ $(n \geq 1)$. Moreover, every pseudo-absorbing primary multiplication $R$-module apart from $R$ is pure-injective.
\end{cor}

\begin{proof}
{ Suppose that $R$ is a  discrete valuation domain with  maximal ideal $P = Rp$, and $M$ is a pseudo-absorbing primary multiplication $R$-module. Suppose also that $N$ is the indecomposable summand of $M$. It follows from Proposition \ref{242} (2) that $N$ is an indecomposable pseudo-absorbing primary multiplication module, and so $N \cong R/P^n$ $(n \geq 1)$, by Theorem \ref{prev}. For the``Moreover" statement, since $M$ is a direct sum of copies of $R/P^n$, every pseudo-absorbing primary multiplication $R$-module apart from $R$ is pure-injective, by Theorem \ref{prev} and \cite[Proposition 1.3]{Ebrahimipure}.
}
\end{proof}

\begin{center}{\section{The separated pseudo-absorbing primary multiplication modules over a pullback ring
}}\end{center}

In this section, the indecomposable separated pseudo-absorbing primary multiplication modules over a pullback of two local Dedekind domains are described and some of their basic properties are given. Recall the construction: if $R_1$ and $R_2$ are two  discrete valuation domains with maximal ideals $P_1 = Rp_1$ and $P_2 = Rp_2$ respectively, then 
\begin{equation}\label{eq:1}
R = (R_1 \xrightarrow{\mathcal{V}_1(x)} \overline{R} \xleftarrow{\mathcal{V}_2(x)} R_2)
\end{equation}
 is the pullback ring of $R_1$ and $R_2$ with a unique maximal ideal $P = P_1\oplus P_2$. Note that  $\overline{R} \cong R/P \cong R_1/P_1\cong R_2/P_2$  is a field. This means that $R$ is a Noetherian local ring with unique maximal ideal $P$. So it is easy to see that $P_1$ and $P_2$ are the other prime ideals of $R$. For more details on the pullback-theoretic of two discrete valuation domains, see \cite{Ebrahimipure}.
 
 In the following results of this paper we assume that $R$ is a pullback ring as in \eqref{eq:1} and $S = (S_1 = S/P_2S \xrightarrow{f_1(x)} \overline{S} = S/PS \xleftarrow{f_2(x)} S_2 = S/P_1S)$ is a separated $R$-module.

\begin{prop} \label{3.1.1}
Let the situation be as in \eqref{eq:1}. Then the following statements hold:

$(1)$ The ideals $\{0\}$, $0\oplus P_2^m$, $P_1^n\oplus 0$ and $P_1^n \oplus P_2^m$ are $2$-absorbing primary ideals of $R$ for every $m,n\geq 1$.

$(2)$ Suppose that $T = (T_1, T_2)$ is a pseudo-absorbing primary submodule of the non-zero separated $R$-module $S = (S_1 = S/P_2S \xrightarrow{f_1(x)} \overline{S} = S/PS \xleftarrow{f_2(x)} S_2 = S/P_1S)$. Then $T_1$ and $T_2$ are pseudo-absorbing primary submodules of $S_1$ and $S_2$, respectively.

$(3)$ Suppose that $T = (T_1, T_2)$ is a pseudo-absorbing primary submodule of the non-zero separated $R$-module $S = (S_1 = S/P_2S \xrightarrow{f_1(x)} \overline{S} = S/PS \xleftarrow{f_2(x)} S_2 = S/P_1S)$. Then $(T :_R S) \in \{0, 0\oplus P_2^m, P_1^n\oplus 0, P_1^n \oplus P_2^m\}$ for some $m,n\geq 1$.

$(4)$ Suppose that $T = (T_1, T_2)$ is a pseudo-absorbing primary submodule of the non-zero separated $R$-module $S = (S_1 = S/P_2S \xrightarrow{f_1(x)} \overline{S} = S/PS \xleftarrow{f_2(x)} S_2 = S/P_1S)$ such that  $S_1$ and $S_2$ are pseudo-absorbing primary multiplication $R_1$-module and $R_2$-module, respectively. Then $(T :_R S) = 0$ implies $T = 0$ and $S_i \neq 0$, for each $i = 1, 2$. 

$(5)$ Let $S = (S_1 = S/P_2S \xrightarrow{f_1(x)} \overline{S} = S/PS \xleftarrow{f_2(x)} S_2 = S/P_1S)$ be a non-zero separated pseudo-absorbing primary multiplication $R$-module. Then $S$ contains at least one pseudo-absorbing primary submodule.
\end{prop}
\begin{proof}
{$(1)$ It follows from \cite[Proposition 3.1]{Dolati} that $\{0\}$ is a $2$-absorbing primary ideal of $R$. Since $\sqrt{0\oplus P_2^m} = 0\oplus P_2$, $\sqrt{P_1^n\oplus 0} = P_1\oplus 0$ and $\sqrt{P_1^n\oplus P_2^m} = P_1\oplus P_2$ are prime ideals of $R$, we conclude by \cite[Theorem 2.8]{Badawi2} that $0\oplus P_2^m$, $P_1^n\oplus 0$ and $P_1^n \oplus P_2^m$ are $2$-absorbing primary ideals of $R$ for every $m,n\geq 1$.

$(2)$ Suppose that $T = (T_1, T_2)$ is a pseudo-absorbing primary submodule of the non-zero separated $R$-module $S = (S_1 = S/P_2S \xrightarrow{f_1(x)} \overline{S} = S/PS \xleftarrow{f_2(x)} S_2 = S/P_1S)$. Suppose also that $a,b,c \in R_1$ and $abc \in (T_1 :_{R_1} S_1)$. Without loss of generality we may assume that $a,b$ and $c$ are nonunit elements. Then $a,b,c\in P_1$, and so $\mathcal{V}_1(abc) = 0 = \mathcal{V}_2(0)$ implies $(abc, 0)\in R$. We show that $(abc, 0)\in (T :_R S)$. Suppose that $(s_1, s_2) \in S$. Then $(abcs_1, 0)\in T$, since $abcs_1 \in T_1\cap P_1S$, $0 \in T_2\cap P_2S$ and $f_1(abcs_1) = 0 = f_2(0)$. Hence $(abcs_1, 0) = (a, 0)(b, 0)(c, 0)(s_1, s_2)\in T$. Thus $(abc, 0)\in (T :_R S)$. But $(T :_R S)$ is a $2$-absorbing primary ideal of $R$, hence either $(a, 0)(b, 0)\in (T :_R S)$ or $(a, 0)(c, 0)\in \sqrt{(T :_R S)}$ or $(b, 0)(c, 0)\in \sqrt{(T :_R S)}$. Therefore either $ab\in (T_1 :_{R_1} S_1)$ or $ac\in \sqrt{(T_1 :_{R_1} S_1)}$ or $bc\in \sqrt{(T_1 :_{R_1} S_1)}$. Thus $T_1$ is a pseudo-absorbing primary submodule of $S_1$. A similar argument shows that $T_2$ is a pseudo-absorbing primary submodule of $S_2$.

$(3)$ Suppose that $T = (T_1, T_2)$ is a pseudo-absorbing primary submodule of the non-zero separated $R$-module $S = (S_1 = S/P_2S \xrightarrow{f_1(x)} \overline{S} = S/PS \xleftarrow{f_2(x)} S_2 = S/P_1S)$. It follows from \cite[Theorem 2.11]{Badawi2} that  $2$-absorbing primary ideals of $R_1$ are of the form $\{0\}$ and $P_1^n$ for some positive integer $n$ and  $2$-absorbing primary ideals of $R_2$ are of the form $\{0\}$ and $P_2^m$ for some positive integer $m$. Hence $(T :_R S) = (T_1 :_{R_1} S_1)\times (T_2 :_{R_2} S_2)$ gives $(T :_R S) = 0$, $0\oplus P_2^m$, $P_1^n\oplus 0$ or $P_1^n \oplus P_2^m$ for some $m,n\geq 1$.

$(4)$ Suppose that $T = (T_1, T_2)$ is a pseudo-absorbing primary submodule of the non-zero separated $R$-module $S = (S_1 = S/P_2S \xrightarrow{f_1(x)} \overline{S} = S/PS \xleftarrow{f_2(x)} S_2 = S/P_1S)$ such that  $S_1$ and $S_2$ are pseudo-absorbing primary multiplication $R_1$-module and $R_2$-module, respectively. Since $T_1$ and $T_2$ are pseudo-absorbing primary submodules of $S_1$ and $S_2$, respectively by part $(2)$, we have $T_1 = (T_1 :_{R_1} S_1)S_1$ and $T_2 = (T_2 :_{R_2} S_2)S_2$, and so it is enough to show that $(T_1 :_{R_1} S_1)$ and $(T_2 :_{R_2} S_2)$ are zero. Suppose to the contrary, $(T_1 :_{R_1} S_1)\neq 0$ and since $S$ is nonzero, we may assume that $S_1\neq 0$. Hence there exists a nonzero element $r_1$ of $(T_1 :_{R_1} S_1)\subseteq P_1$, and so $\mathcal{V}_1(r_1) = 0 = \mathcal{V}_2(0)$ implies $(r_1, 0) \in R$. Therefore $(r_1 , 0)(s_1, s_2) = (r_1s_1, 0) \in T$ for every $(s_1, s_2) \in S$, because $r_1s_1\in T_1\cap P_1S$, $0 \in T_2\cap P_2S$, $f_1(r_1s_1) = 0 = f_2(0)$ and $T$ is separated. Thus  $0 \neq (r_1, 0)S \subseteq T$, a contradiction. We show that $S_2\neq 0$, for otherwise it would follow that $S = (P_1 \oplus 0)S$, and so $(0 \oplus P_2)S = (0 \oplus P_2)(P_1 \oplus 0)S = 0$. This means that $0\neq (0 \oplus P_2)\subseteq (0 :_R S)\subseteq (T :_R S)$, a contradiction. Therefore, by a similar argument as above we can see that $(T_2 :_{R_2} S_2) = 0$. Thus $T_1 = T_2 = 0$ and $S_i \neq 0$, for each $i = 1, 2$.

$(5)$ Let $S = (S_1 = S/P_2S \xrightarrow{f_1(x)} \overline{S} = S/PS \xleftarrow{f_2(x)} S_2 = S/P_1S)$ be a non-zero separated pseudo-absorbing primary multiplication $R$-module. It follows from Theorem \ref{3.2.3} that $S_1$ is a pseudo-absorbing primary multiplication $R_1$-module. We show that there exists a pseudo-absorbing primary submodule of $S_1$. Suppose to the contrary, there is no pseudo-absorbing primary submodule of $S_1$. Since every prime submodule is a pseudo-absorbing primary submodule of $S_1$, it can be seen by \cite[Lemma 1.3, Proposition 1.4]{MacCasland} that $S_1$ is a torsion divisible $R_1$-module with $P_1S_1 = S_1$ and $S_1$ is not finitely generated. By a similar argument as in \cite[Proposition 2.7 Case 2]{Ebrahimi}, we get $S_1 \cong E(R_1/P_1)$, a contradiction, by Proposition \ref{2.5.5}. Therefore there exists a pseudo-absorbing submodule $T_1$ of $S_1$. Now, let $g$ be the projection map of $R$ onto $R_1$ and apply Proposotion \ref{242} to this particular choice and obtain that $T_1$ is a pseudo-absorbing primary $R$-submodule of $S_1 = S/P_2S$. Thus there exists at least one pseudo-absorbing primary submodule of $S$, by Proposition \ref{2.3} $(1)$.
}
\end{proof}

\begin{thm} \label{3.2.3}
Let $S = (S_1 = S/P_2S \xrightarrow{f_1(x)} \overline{S} = S/PS \xleftarrow{f_2(x)} S_2 = S/P_1S)$ be a non-zero separated $R$-module. Then $S$ is a pseudo-absorbing primary multiplication $R$-module if and only if $S_1$ and $S_2$ are pseudo-absorbing primary multiplication $R_1$-module and $R_2$-module, respectively.
\end{thm}
\begin{proof}
{ First suppose that $S$ is a separated pseudo-absorbing primary multiplication $R$-module. If $\overline{S} = 0$, then \cite[Lemma 2.7 (i)]{Ebrahimipure} gives $S = S_1\oplus S_2$. Thus $S_1$ and $S_2$ are pseudo-absorbing primary multiplication $R_1$-module and $R_2$-module, respectively, by Proposition \ref{242} $(2)$. If $\overline{S} \neq 0$, again by Proposition \ref{242} $(1)$, we have $S_1 \cong S/(0 \oplus P_2)S$ is a pseudo-absorbing primary multiplication $R/(0 \oplus P_2)\cong R_1$-module, since $(0 \oplus P_2)\subseteq ((0 \oplus P_2)S : S)$. By a similar argument, one may see that $S_2$ is a pseudo-absorbing primary multiplication $R/(P_1 \oplus 0)\cong R_2$-module. Now, assume that $S_1$ and $S_2$ are pseudo-absorbing primary multiplication $R_1$-module and $R_2$-module, respectively. Suppose that $T$ is a pseudo-absorbing primary submodule of $S$. It follows from Proposition \ref{3.1.1} $(3)$ that $(T :_R S) = 0$ or $0\oplus P_2^m$ or $P_1^n\oplus 0$ or $P_1^n \oplus P_2^m$ for some $m,n\geq 1$. We distinguish three cases:
\newline {\bf Case one:} If $(T :_R S) = 0$, then $T = (T :_R S)S = 0$, by Proposition \ref{3.1.1} $(3)$, and so $S$ is a pseudo-absorbing primary multiplication $R$-module.
\newline {\bf Case two:} If $(T :_R S) = P_1^n \oplus P_2^m$ for some $m,n\geq 1$, then \cite[Proposition 4.2 (i)]{Ebrahimisemiprime} gives $(T_1 :_{R_1} S_1) = P_1^n$ and $(T_2 :_{R_2} S_2) = P_2^m$. Now apply Proposition \ref{3.1.1} $(2)$ to obtain $T_1 = (T_1 :_{R_1} S_1)S_1 = P_1^nS_1$ and $T_2 = (T_2 :_{R_2} S_2)S_2 = P_2^mS_2$. To complete the proof, it is only necessary to observe that $T = (P_1^n\oplus P_2^m)S$. Clearly $ (P_1^n\oplus P_2^m)S \subseteq T$. Suppose that $t = (t_1, t_2)\in T$. Then there are $s_1\in S_1$ and $s_2 \in S_2$ such that $t_1 = p_1^ns_1$ and $t_2 = p_2^ms_2$. But $s_1 \in S_1$ and $s_2 \in S_2$, hence there are $u_1\in S_1$ and $u_2 \in S 2$ such that $(s_1, u_2), (u_1, s_2)\in S$. This means that $t = (p_1^ns_1, p_2^ms_2) = (p_1^n, 0)(s_1, u_2) + (0, p_2^m)(u_1, s_2) \in (P_1^n\oplus P^m_2)S$, and thus $T = (P_1^n\oplus P^m_2)S$.
\newline {\bf Case three:} If $(T :_R S) = P_1^n \oplus 0$ for some $n \geq 1$, then \cite[Proposition 4.2 (i)]{Ebrahimisemiprime} gives $(T_1 :_{R_1} S_1) = P_1^n$ and $(T_2 :_{R_2} S_2) = 0$. Now apply Proposition \ref{3.1.1} $(2)$ to obtain $T_1 = (T_1 :_{R_1} S_1)S_1 = P_1^nS_1$, $T_2 = (T_2 :_{R_2} S_2)S_2 = 0$. To complete the proof, it is only necessary to observe that $T = (P_1^n\oplus 0)S$. Clearly $ (P_1^n\oplus 0)S \subseteq T$. Suppose that $t = (t_1, 0)\in T$. Then there are $s_1\in S_1$ and $s_2 \in S_2$ such that $t_1 = p_1^ns_1$ and $(s_1, s_2)\in S$. Hence $t = (p_1^ns_1, 0) = (p_1^n, 0)(s_1, s_2) \in (P_1^n\oplus 0)S$, and thus $T = (P_1^n\oplus 0)S$. By a similar argument as above, we can see that $T = (0\oplus P_2^m)S$ if $(T :_R S) = 0\oplus P_2^m$ for some $m \geq 1$. Therefore $S$ is a pseudo-absorbing primary multiplication $R$-module.

}
\end{proof}

Consequently, in view of Theorems \ref{prev} and \ref{3.2.3}, we state the following corollary.

\begin{cor}\label{3.3.3}
Suppose that $R$ is a pullback ring as described in \eqref{eq:1}. Then the following separated $R$-modules are indecomposable and pseudo-absorbing primary multiplication modules:

$(1)$ $R$;

$(2)$ $S = (S_1 = R_1/P_1^n \xrightarrow{f_1(x)} \overline{R} \xleftarrow{f_2(x)} R_2/P_2^m)$;

$(3)$ $S = (S_1 = R_1 \xrightarrow{f_1(x)} \overline{R} \xleftarrow{f_2(x)} R_2/P_2^m)$;

$(4)$ $S = (S_1 = R_1/P_1^n \xrightarrow{f_1(x)} \overline{R} \xleftarrow{f_2(x)} R_2)$;

 for every $m,n\geq 1$.

\end{cor}

\begin{proof}
{It follows from Theorems \ref{prev} and \ref{3.2.3} that these $R$-modules are pseudo-absorbing primary multiplication modules and from \cite [Lemma 2.8]{Ebrahimipure} that they are indecomposable.
}
\end{proof}

The observation in Corollary \ref{3.3.3} provides a clue to a characterization of non-zero indecomposable separated pseudo-absorbing primary multiplication $R$-modules.

\begin{thm}\label{3.3.4}
Suppose that $R$ is a pullback ring as described in \eqref{eq:1}. Then every non-zero indecomposable separated pseudo-absorbing primary multiplication $R$-module is isomorphic to one of the modules listed in {\rm Corollary \ref{3.3.3}}.
\end{thm}

\begin{proof}
{Suppose that $S = (S_1 = S/P_2S \xrightarrow{f_1(x)} \overline{S} = S/PS \xleftarrow{f_2(x)} S_2 = S/P_1S)$ is a non-zero indecomposable separated pseudo-absorbing primary multiplication $R$-module such that $S\neq R$. It follows from Proposition \ref{3.1.1} $(5)$ that there exists at least one pseudo-absorbing primary submodule of $S$, say $T = (T_1, T_2)$, and so $T_i$ is a pseudo- absorbing submodule $S_i$, by Proposition \ref{3.1.1} $(2)$. If $T_i = S_i$ for some $i\in \{1,2\}$, say $T_1 = S_1$, then $(T_1 :_{R_1} S_1) = R_1$ is a pseudo-absorbing primary ideal of $R_1$, which contradicts \cite[Theorem 2.11]{Badawi2}. First we show that $S\neq PS$. As remarked in Proposition \ref{3.1.1} $(3)$, consider various cases for $(T :_R S)$. 
{\bf Case one:} If $(T :_R S) = 0$, then $S\neq PS$ and \cite [Lemma 2.7 (i)]{Ebrahimipure} implies $S = S_1\oplus S_2$. Since $S$ is indecomposable, we conclude that either $S_1 = 0$ or $S_2 = 0$, and this contradicts Proposition \ref{3.1.1} $(4)$.
{\bf Case two:} If $(T :_R S) = P_1^n \oplus P_2^m$ for some $m,n\geq 1$, then $T = (T :_{R} S)S = (P_1^n \oplus P_2^m)S$, since $S$ is a pseudo-absorbing primary multiplication $R$-module. This means that $T_1 = (T_1 :_{R_1} S_1)S_1 = P_1^nS_1 \subseteq P_1S_1 = PS/P_2S \neq S/P_2S = S_1$, and thus $PS \neq S$. 
{\bf Case three:} If $(T :_R S) = P_1^n \oplus 0$ for some $n \geq 1$, then $T = (T :_{R} S)S = (P_1^n \oplus 0)S$, as $S$ is a pseudo-absorbing primary multiplication $R$-module. This means that $T_1 = (T_1 :_{R_1} S_1)S_1 = P_1^nS_1 \subseteq P_1S_1 = PS/P_2S \neq S/P_2S = S_1$, and thus $PS \neq S$. Similarly, if $(T :_R S) = 0 \oplus P_2^m$ for some $m \geq 1$, we get $PS \neq S$. Assume that $s \in S_1 \cup S_2$ such that $s\notin PS$ and 
\begin{center}
$o(s) =$
$\begin{cases}
{\rm least}$ $m$ ${\rm such}$ ${\rm that}$ $ P^ms$ ${\rm =}$ $0 & \text{if  there is such $m$},\\
\infty & \text{if there is no such $m$}.
\end{cases}$
\end{center}
We can assume that $s\in S_2$, and hence there exists $s_1\in S_1$ such that $(s_1, s =: s_2) \in S$, $f_1(s_1) = f_1(s_2)$ and $o(s_1) = n$ is minimal. By a similar rgument as in \cite[Theorem 2.9]{Ebrahimipure}, we can see that  $R_is_i$ is pure in $S_i$ for $i = 1,2$. Thus, $R_1s_1 \cong R_1/P_1^n$ and $R_2s_2 \cong R_2/P_2^m$ which are direct summands in $S_1$ and $S_2$, respectively. If $o(s_1) = \infty$ and $o(s_2) = \infty$, again by \cite[Theorem 2.9]{Ebrahimipure}, $R_1s_1 \cong R_1$ and $R_2s_2 \cong R_2$ which are pure of $S_1$ and $S_2$, respectively. Now suppose that $\overline{M}$ is the $\overline{R}$-subspace of $\overline{S}$ which is generated by $\overline{s} = f_1(s_1) = f_1(s_2)$. Then $\overline{M} \cong \overline{R}$. So that, on taking $M = (R_1s_1 = M_1 \rightarrow \overline{M} \leftarrow M_2 = R_2s_2)$, we deduce that $M$ is a pseudo-absorbing primary multiplication $R$-submodule of $S$, by Corollary \ref{3.3.3} which is a direct summand of $S$, and since $S$ is indecomposable, we get $S = M$. Thus $S$ is one of the modules listed $(2)$--$(4)$ in the Corollary \ref{3.3.3} and this completes the proof.
}
\end{proof}

In view of Theorem \ref{3.3.4}, we have the following result.

\begin{cor}\label{3.5.5}
Suppose that $R$ is a pullback ring as described in \eqref{eq:1}. Then every non-zero indecomposable separated pseudo-absorbing primary multiplication $R$-module non isomorphic to $R$ is pure-injective. Furthermore, every non-zero indecomposable separated pseudo-absorbing primary multiplication $R$-module non isomorphic to $R$ has a finite-dimensional top.
\end{cor}

\begin{proof}
{ It follows from Theorem \ref{3.3.4} and  \cite[Theorem 2.9]{Ebrahimipure} that every non-zero indecomposable separated pseudo-absorbing primary multiplication $R$-module non isomorphic to $R$ is pure-injective. For the``Furthermore" statement, suppose that $S = (S_1 = S/P_2S \xrightarrow{f_1(x)} \overline{S} = S/PS \xleftarrow{f_2(x)} S_2 = S/P_1S)$ is a non-zero indecomposable separated pseudo-absorbing primary multiplication $R$-module such that $S\neq R$. Then $S_1$ and $S_2$ are pseudo-absorbing primary multiplication $R_1$-module and $R_2$-module, respectively, by Theorem \ref{3.2.3}. Suppose also that $T(S_i)$ is the torsion submodule of $S_i$ such that $T(S_i) \neq S_i$ for $i = 1,2$. Then $T(S_i)$ is a prime submodule of $S_i$ for $i = 1,2$, and so is a  pseudo-absorbing primary submodule of $S_i$ such that $(T(S_i) :_{R_i} S_i) = 0$ for $i = 1,2$ (see \cite[Lemma 3.8]{Moore}). Thus $T(S_i) = 0$  for each $i = 1,2$, and hence $S_i$ is either torsion or torsion-free for each $i = 1,2$. Now apply Theorem \ref{3.3.4} and \cite[Theorem 3.14]{Esmaeili} to see that every non-zero indecomposable separated pseudo-absorbing primary multiplication $R$-module non isomorphic to $R$ has a finite-dimensional top.
}
\end{proof}

\begin{center}{\section{The non-separated pseudo-absorbing primary multiplication modules over a pullback ring
}}\end{center}

Let the situation be as in \eqref{eq:1}, so that $R$ is a pullback ring as in \eqref{eq:1} and $S = (S_1 = S/P_2S \xrightarrow{f_1(x)} \overline{S} = S/PS \xleftarrow{f_2(x)} S_2 = S/P_1S)$ is a separated $R$-module. In this section, the indecomposable non-separated pseudo-absorbing primary multiplication modules with finite-dimensional top over a pullback of two local Dedekind domains are described.  In fact, it turns out that non-separated pseudo-absorbing primary multiplication modules with finite-dimensional top can be obtained by amalgamating finitely many separated indecomposable pseudo-absorbing primary multiplication modules.

Suppose that $M$ is an indecomposable non-separated pseudo-absorbing primary multiplication $R$-module
with finite-dimensional top and the exact sequence $0 \rightarrow K \xrightarrow{i} S \xrightarrow{\varphi} M \rightarrow 0$ is a separated representation of $M$. We show that $S$ is a direct sum of only finitely many indecomposable pseudo-absorbing primary multiplication modules. First we show that 
$M$ is a pseudo-absorbing primary multiplication $R$-module if and only if $S$ is a pseudo-absorbing primary multiplication $R$-module. 

\begin{prop} \label{411}
Suppose that $R$ is a pullback ring as described in \eqref{eq:1} and $0 \rightarrow K \xrightarrow{i} S \xrightarrow{\varphi} M \rightarrow 0$ is a separated representation of $M$. If $M$ is non-separated, then $(T :_R S) \notin\{0, P_1^n \oplus 0, 0\oplus P_2^m\}$ for each submodule $T$ of $S$ and $m,n \geq 1$. 
\end{prop}
\begin{proof}
{Let $T$ be a submodule of $S$ such that $(T :_R S) \in\{0, P_1^n \oplus 0, 0\oplus P_2^m\}$ for some $m,n \geq 1$. First, suppose that $(T :_R S) = 0$, and so $(0 :_R S) = 0$. Let $r = (r_1, r_2)\in (0 :_R M)$. Then $rM = rS/K = 0$, which means that $rS \subseteq K$. Hence $rPS \subseteq PK = 0$, by \cite[Proposition 2.4]{Levy}, and so $rP = 0$. Therefore $r_1p_1 = 0 = r_2p_2$; thus $r = (0, 0)$, and so $(0 :_R M) = 0$. This implies that $(P_1\oplus 0)\nsubseteq (0 :_R M)$ and $(0\oplus P_2)\nsubseteq (0 :_R M)$. Without loss of generality, assume that $(P_1\oplus 0)M \cap (0\oplus P_2)M\neq 0$; hence there exists a non-zero $m\in (P_1\oplus 0)M \cap (0\oplus P_2)M$. But $\varphi$ is an $R$-homomorphism of $S$ onto $M$, then there exists $x \in S$ such that $m = \varphi(x)$. Since $\varphi^{-1}(Rm) = \varphi^{-1}(\varphi(Rx)) = Rx$ and $Rm \neq 0$, it follows from \cite[Lemma 4.1]{Dolati} that $0 \rightarrow K \rightarrow Rx \rightarrow Rm \rightarrow 0$ is a separated representation of $Rm$ such that $K \subseteq PRx$. Therefore, there exist $m_1, m_2 \in M$ and $s,t\in \mathbb{N}$ such that $m = (p_1^s, 0)m_1 = (0, p_2^t)m_2$. It follows that $Pm = (P_1\oplus P_2)m = 0$, and hence $\varphi(Px) = 0$. Thus $\varphi(P_1x) = \varphi(P_2x) = 0$. Since $\varphi$ is one-to-one on $P_1S$ and on $P_1S$, we conclude that $Px = 0$. Therefore $K \subseteq PRx = 0$, and hence $M \cong S$ is a separated $R$-module which is a contradiction. So suppose that $(T :_R S) = P_1^n \oplus 0$ for some $n \geq 1$. Then $(0 :_R S) \subseteq (T :_R S)$ implies that either $(0 :_R S) = 0$ or $(0 :_R S) = P_1^{n_1} \oplus 0$ for some $n_1 \geq n$. The first case follows from the above comment. Assume that $(0 :_R S) = P_1^{n_1} \oplus 0$ for some $n_1 \geq n$. Then $(P_1^{n_1}\oplus 0)S = 0$, which implies that $P_1^{n_1}\oplus 0 \subseteq ((0\oplus P_2)S :_R S)$. Since $0\oplus P_2 \subseteq ((0\oplus P_2)S :_R S)$, we conclude that $P_1^{n_1}\oplus  P_2 \subseteq ((0\oplus P_2)S :_R S)$. Therefore, by \cite[Proposition 4.2]{Dolati}, we get  $K \subseteq P^{n_1}S \subseteq (P_1^{n_1} \oplus P_2)S\subseteq (0\oplus P_2)S$. Thus $K = 0$ because $K \cap (0 \oplus P_2)S = 0$, by \cite[Proposition 2.3]{Levy}, and hence $M \cong S$ is a separated $R$-module which is a contradiction. A similar argument shows that $(T :_R S) \neq 0\oplus P_2^m$ for every $m \geq 1$, and the proof is complete.
}
\end{proof}

\begin{thm}\label{422}
Suppose that $R$ is a pullback ring as described in \eqref{eq:1} and $0 \rightarrow K \xrightarrow{i} S \xrightarrow{\varphi} M \rightarrow 0$ is a separated representation of $M$ such that $M$ is non-separated $R$-module. Then $M$ is a pseudo-absorbing primary multiplication $R$-module if and only if $S$ is a pseudo-absorbing primary multiplication $R$-module.
\end{thm}
\begin{proof}
{First, suppose that $S$ is a pseudo-absorbing primary multiplication $R$-module. Since $S\cong M/K$, the proof follows from Proposition \ref{2.3} $(1)$. We now turn to the converse statement, which is the non-trivial part of
the result. Let $M$ be a pseudo-absorbing primary multiplication $R$-module and $T$ a pseudo-absorbing primary submodule of $S$. It follows from Proposition \ref{411} that $(T :_R S) \notin\{0, P_1^n \oplus 0, 0\oplus P_2^m\}$ for each submodule $T$ of $S$ and $m,n \geq 1$, as $M$ is non-separated, and so Proposition \ref{3.1.1} $(3)$ gives $(T :_R S) = P_1^n \oplus P_2^m$ for some $m,n\geq 1$. Assume that $n\geq m$. Therefore, by \cite[Proposition 4.2]{Dolati}, we get $K \subseteq P^nS \subseteq (P_1^n \oplus P_2^m)S\subseteq T$. By using the fact that $T/K$ is a pseudo-absorbing primary submodule of $S/K \cong M$ and $M$ is pseudo-absorbing primary multiplication, we can see that $T/K = (T/K :_R S/K)S/K = (P_1^n \oplus P_2^m)(S/K) = ((P_1^n \oplus P_2^m)S + K)/K = ((P_1^n \oplus P_2^m)S)/K$, and thus $T = (P_1^n \oplus P_2^m)S$. Hence $S$ is a pseudo-absorbing primary multiplication $R$-module.
}
\end{proof}

\begin{cor}\label{inj}
Suppose that $R$ is a pullback ring as described in \eqref{eq:1} and $0 \rightarrow K \xrightarrow{i} S \xrightarrow{\varphi} M \rightarrow 0$ is a separated representation of $M$ such that $M$ is an indecomposable non-separated pseudo-absorbing primary multiplication $R$-module with finite-dimensional top over $\overline{R}$. Then $S$ is a pure-injective $R$-module with finite-dimensional top.
\end{cor}
\begin{proof}
{By Theorem \ref{422}, $S$ is a pseudo-absorbing primary multiplication $R$-module. Hence $S$ is a pure-injective $R$-module, by Corollary \ref{3.5.5}. It follows from \cite[Proposition 2.6 (i)]{Ebrahimipure} that $S/PS\cong M/PM$; thus $S$ is finite-dimensional top.
}
\end{proof}

Our next goal is to show that if $M$ is an indecomposable non-separated pseudo-absorbing primary multiplication $R$-module
with finite-dimensional top and the exact sequence $0 \rightarrow K \xrightarrow{i} S \xrightarrow{\varphi} M \rightarrow 0$ is a separated representation of $M$, then $S$ is a direct sum of only finitely many indecomposable pseudo-absorbing primary multiplication modules. But it seems that the solving of this problem depends on the results found in \cite[Lemma 3.1, Proposition 3.2 and Proposition 3.4]{Ebrahimipure}. In fact, the proofs of these results are perhaps needed more than the result itself. Note that the key properties in those results are the pure-injectivity of $S$ which is obtained from the pure-injectivity of $M$ (see \cite[Proposition2.6 (ii)]{Ebrahimipure}), the indecomposability and the nonseparability of $M$. But in this paper the pure-injectivity of $S$ is obtained from the pseudo-absorbing primary multiplicity of $M$ (see Corollary \ref{inj}). Therefore, we can replace the statement ``$M$ is pure-injective" by ``$M$ is pseudo-absorbing primary multiplication". However, there is one beautiful existence result concerning the direct summands of $S$.

\begin{cor}\label{444}
Suppose that $R$ is a pullback ring as described in \eqref{eq:1} and $0 \rightarrow K \xrightarrow{i} S \xrightarrow{\varphi} M \rightarrow 0$ is a separated representation of $M$ such that $M$ is an indecomposable non-separated pseudo-absorbing primary multiplication $R$-module with finite-dimensional top over $\overline{R}$. Then $S$ is a direct sum of finitely many indecomposable pseudo-absorbing primary multiplication modules.
\end{cor}
\begin{proof} 
{First we show that $R$ does not occur among the direct summands of $S$. Assume that $S = R \oplus N$ for some submodule $N$ of $S$. Then $K \subseteq N$, because Soc$(R) = 0$. Hence $M \cong N/K \oplus R$, a contradiction, since $M$ is indecomposable and non-separated. Therefore, in view of the observations in the above paragraph, $S$ is a direct sum of finitely many indecomposable pseudo-absorbing primary multiplication modules (described in $(2)$--$(4)$ of Corollary \ref{3.3.3}), by Proposition \ref{242} $(2)$ and Theorem \ref{422}. 
}
\end{proof}
\begin{cor} \label{455} Let the situation be as in {\rm Corollary \ref{444}}. Then the following statements hold:

$(1)$ The quotient fields $Q(R_1)$, $Q(R_2)$ do not occur among the direct summands of $S$.

$(2)$ At most two copies of modules of infinite length can occur among the indecomposable summands of $S$.
\end{cor}

We are now ready for the main result of this paper. But before starting let us quickly describe its idea. Let the situation be as in \eqref{eq:1} and let $0 \rightarrow K \xrightarrow{i} S \xrightarrow{\varphi} M \rightarrow 0$ be a separated representation of an $R$-module $M$.  It follows from Corollary \ref{444} that if $M$ is indecomposable non-separated pseudo-absorbing primary multiplication  with finite-dimensional top over $\overline{R}$, then $S$ is a direct sum of finitely many indecomposable pseudo-absorbing primary multiplication modules, and we have shown that these modules are as in Corollary \ref{3.3.3}. It should be observed that $R$, $Q(R_1)$ and $Q(R_2)$ (the quotient fields of $R_1$ and $R_2$) cannot be amalgamated with any other direct summands of $S$, and hence cannot take place in any separated representation $0 \rightarrow K \xrightarrow{i} S \xrightarrow{\varphi} M \rightarrow 0$ because $(0 :_R ker(\varphi)) = (0 :_R K) = P$ which means that $K\subseteq {\rm Soc}(S)$, and so $M$ is obtained by amalgamation in the socles of the various direct summands of $S$ whereas $0 = {\rm Soc}(R) = {\rm Soc}(Q(R_1)) = {\rm Soc}(Q(R_2))$ and this clarify the proofs of Corollies \ref{444} and \ref{455}. Thus natural questions of these sort are: do these conditions give us extra conditions on the possible direct summands of $S$? How can we get  the direct summands of $M$ using the amalgamated summands of $S$? The goal of this paper is to answer these questions in the case in which $M$ is indecomposable non-separated pseudo-absorbing primary multiplication module with finite-dimensional top. Previously, these questions are answered by Levy in \cite{Levy2,Levy3} for the case of finitely generated $R$-modules $M$. We recall what we will need from that. In fact, Levy showed that the indecomposable finitely generated $R$-modules are of two non-overlapping types which he called them deleted cycle and block cycle types. We observe that deleted cycle modules are very pertinent to us. We get such modules from a direct summand, $S$, of indecomposable separated modules by amalgamating the direct summands of $S$ in pairs to form a chain but leaving the two ends unamalgamated. By considering the fact that the dimension of the socle of any finitely generated indecomposable separated module over $R$ is smaller than two, each indecomposable summand of $S$ may be amalgamated with at most two other indecomposable summands. We give an example of a non-separated module which is obtained by amalgamating the summands of a separated $R$-module $S$:
\begin{center}
$S_1 = (R_1/P_1^5 \rightarrow \overline{R} \leftarrow R_2/P_2^3) = Ra$ with $P_2^3a = 0 = P_1^5a$

$S_2 = (R_1/P_1^4 \rightarrow \overline{R} \leftarrow R_2/P_2^2) = Ra'$ with $P_1^4a' = 0 = P_2^2a'$
\end{center}
and so we can form the non-separated module $(S_1 \oplus S_2)/(R(p_2^2a - p_1^3a') = Rc +Rc'$ such that $c = a + R(p_2^2a - p_1^3a')$, $c' = a' + R(p_2^2a - p_1^3a')$ and $P_2^3c = 0 = P_1^4c' = P_2^2c = P_1^3c'$, which is obtained by identifying the ``$P_2$-part'' of the socle of the separated module $S_1$ with the ``$P_1$-part'' of the socle of the separated module $S_2$. See the following schematic version of diagram as depicted in Figure~\ref{fig:md1}:
\begin{figure}
    \begin{tikzpicture}[
            > = stealth, 
            shorten > = 1pt, 
            auto,
            node distance = 2cm, 
            semithick, 
            targetnode/.style={circle, draw=red!60, fill=white!5, very thick, minimum size=2mm},
            entrynode/.style={circle, draw=blue!60, fill=white!5, very thick, minimum size=2mm},
            scale=0.6,
        ]

        \tikzstyle{every state}=[
            draw = black,
            thick,
            fill = white,
            scale=0.6
        ]

        \node[entrynode] (a) {$a$};
        \node[state] (u2) [above right of=a] {$p_1a$};
        \node[state] (u6) [above right of=u2] {$p_1^2a$};
        \node[state] (u9) [above right of=u6] {$p_1^3a$};
        \node[state] (u10) [above right of=u9] {$p_1^4a$};
        \node[state] (u11) [above right of=u10] {$0$};
        \node[state] (u3) [right of=a] {$p_2a$};     
        \node[state] (u7) [right of=u3] {$p_2^2a$};
        \node[state] (u12) [right of=u7] {$0$};
        \node[state] (u15) [below of=u7] {$p_1^3a'$};
        \node[state] (u14) [left of=u15] {$p_1^2a'$};
        \node[state] (u4) [left of=u14] {$p_1a'$};
        \node[state] (u16) [right of=u15] {$0$};
        \node[targetnode] (a') [left of=u4] {$a'$};
        \node[state] (u17) [below right of=a'] {$p_2a'$};
        \node[state] (u18) [below right of=u17] {$0$};

        \path[->] (a) edge node {} (u2);
        \path[->] (a) edge node {} (u3);
        \path[->] (u6) edge node {} (u9);
        \path[->] (u9) edge node {} (u10);
        \path[->] (u10) edge node {} (u11);
        \path[->] (u7) edge node {} (u12);
        \path[<->] (u7) edge node {} (u15);
        \path[->] (u2) edge node {} (u6);
        \path[<-] (u4) edge node {} (a');
        \path[<-] (u7) edge node {} (u3);
    
    \path[->] (u4) edge node {} (u14);
    \path[->] (u14) edge node {} (u15);
    \path[->] (u15) edge node {} (u16);
    \path[->] (a') edge node {} (u17);
    \path[->] (u17) edge node {} (u18);

    \end{tikzpicture}
\hspace{1cm}
 \begin{tikzpicture}[
            > = stealth, 
            shorten > = 1pt, 
            auto,
            node distance = 1.8cm, 
            semithick, 
            targetnode/.style={circle, draw=red!60, fill=white!5, very thick, minimum size=0.5mm},
            entrynode/.style={circle, draw=blue!60, fill=white!5, very thick, minimum size=0.5mm},
            scale=0.3,
        ]

        \tikzstyle{every state}=[
            draw = black,
            thick,
            fill = white,
            scale=0.3
        ]      

        \node[entrynode] (a) {};
        \node (e) [right of=a]{};
        \node[state] (u2) [above right of=e] {};
        \node[state] (u6) [above right of=u2] {};
        \node[state] (u9) [above right of=u6] {};
        \node[state] (u10) [above right of=u9] {};
        \node (u11) [above right of=u10] {$0$};
        \node[state] (u3) [below right of=e] {};     
        \node[state] (u7) [below right of=u3] {};
        \node (u12) [right of=u7]{$0$};

        \node[state] (u14) [below left of=u7] {};
        \node[state] (u4) [below left of=u14] {};       
        \node (g) [below left of=u4]{};
        \node[targetnode] (a') [left of=g] {};        
        \node[state] (u17) [below right of=g] {};
        \node (u18) [below right of=u17] {$0$};

        \path[-] (a) edge node {} (e);
        \path[-] (e) edge node {} (u2);
        \path[-] (e) edge node {} (u3);
        \path[-] (u6) edge node {} (u9);
        \path[-] (u9) edge node {} (u10);
        \path[-] (u10) edge node {} (u11);
        \path[-] (u7) edge node {} (u12);
        \path[-] (u2) edge node {} (u6);
        \path[-] (u4) edge node {} (g);
        \path[-] (u7) edge node {} (u3);
    
    \path[-] (u4) edge node {} (u14);
    \path[-] (u14) edge node {} (u7);
    \path[-] (a') edge node {} (g);  
    \path[-] (g) edge node {} (u17);
    \path[-] (u17) edge node {} (u18);

    \end{tikzpicture} 
\caption{A schematic version of the diagrams of a non-separated module which is obtained by identifying the ``$P_2$-part'' of the socle of the separated module $S_1$ with the ``$P_1$-part'' of the socle of the separated module $S_2$.}
  \label{fig:md1}
\end{figure}
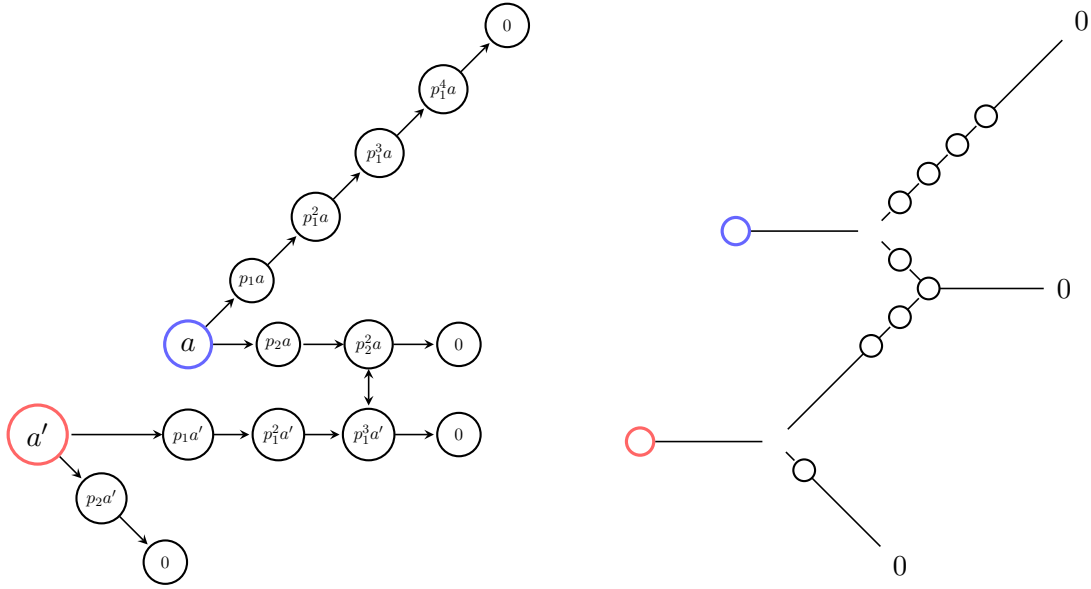

\newpage
The non-zero indecomposable pseudo-absorbing primary multiplication non-separated $R$-modules are obtained in just the same way as indecomposable finitely generated ones; we will use the same detailing but with pseudo-absorbing primary multiplication separated modules in place of the finitely generated ones. We have seen in Corollary \ref{455} that any non-zero indecomposable separated pseudo-absorbing primary multiplication  module with $1$-dimensional socle may occur only at one of the ends of the amalgamation chain, and thus we have to show that the modules obtained by these amalgamations are indecomposable pseudo-absorbing primary multiplication. Hence, the indecomposable pseudo-absorbing primary multiplication non-separated modules with finite-dimensional top will be classified in the following result. 

\begin{thm}\label{main}
Suppose that $R$ is a pullback ring as described in \eqref{eq:1}. Then the indecomposable non-separated pseudo-absorbing primary multiplication $R$-modules $M$  with finite-dimensional top over $\overline{R}$, up to isomorphism, are the following.

$(1)$ The indecomposable modules of finite length (apart from $R/P$ since it is a separated module), i.e, $M = \sum^{s}_{i=1}Ra_i$ such that $p_1^{n_s}a_s = p_2^{m_1}a_1 = 0$, $p_1^{n_i-1}a_i = p_2^{m_{i+1}-1}a_{i+1}$ where $(1 \leq i \leq s - 1)$ and $m_i, n_i \geq 2$ except the case $m_1,n_s\geq 1$.

$(2)$ $M = R/P_1^n + \sum^{s-1}_{i=1}Ra_i$ such that $p_1^{n_{s-1}}a_{s-1} = 0$, $p_1^{n-1}a_0 = p_2^{m_{1}-1}a_{1}$ and $p_1^{n_i-1}a_i = p_2^{m_{i+1}-1}a_{i+1}$ where $R/P_1^n\cong Ra_0$, $(1 \leq i \leq s - 2)$ and $n,m_i, n_i \geq 2$ except the case $n_{s-1}\geq 1$.

$(3)$ $M = \sum^{s-1}_{i=1}Ra_i + R/P_2^n$ such that $p_2^{m_{s-1}}a_{s-1} = 0$, $p_2^{n-1}b_0 = p_1^{n_{1}-1}a_{1}$ and $p_2^{m_i-1}a_i = p_1^{n_{i+1}-1}a_{i+1}$ where $R/P_2^n\cong Rb_0$, $(1 \leq i \leq s - 2)$ and $n,m_i, n_i \geq 2$ except the case $m_{s-1}\geq 1$.

$(4)$  $M = R/P_1^r + \sum^{s-2}_{i=1}Ra_i + R/P_2^n$ such that $p_2^{n-1}b_0 = p_1^{n_{1}-1}a_{1}$, $p_1^{r-1}a_0 = p_2^{m_{s-2}-1}a_{s-2}$  and $p_2^{m_{i}-1}a_{i} = p_1^{n_{i+1}-1}a_{i+1}$ where $R/P_1^r\cong Ra_0$, $R/P_2^n\cong Rb_0$, $(1 \leq i \leq s - 3)$ and $n,m_i, n_i, r \geq 2$ except the case $m_1,n_s\geq 1$.
\end{thm}
\begin{proof}
{Suppose that $R$ is a pullback ring as described in \eqref{eq:1} and $0 \rightarrow K \xrightarrow{i} S \xrightarrow{\varphi} M \rightarrow 0$ is a separated representation of $M$ such that $M$ is indecomposable non-separated pseudo-absorbing primary multiplication $R$-module with finite-dimensional top over $\overline{R}$. Then $S$ is a direct sum of finitely many indecomposable pseudo-absorbing primary multiplication modules, by Corollary \ref{444}. But we have seen already that every indecomposable non-separated pseudo-absorbing primary multiplication $R$-module has one of these forms so it remains to show that the modules obtained by these amalgamation are, indeed, indecomposable pseudo-absorbing primary multiplication $R$-modules. The pseudo-absorbing primary multiplicity of these modules follows from Proposition \ref{2.3} $(1)$ and Theorem \ref{422}, and the indecomposability follows from \cite[1.9]{Levy3} and \cite[Theorem 3.5]{Ebrahimipure}. }
\end{proof}

We close this section with the following corollary.

\begin{cor}
Suppose that $R$ is a pullback ring as described in \eqref{eq:1}. Then every indecomposable pseudo-absorbing primary multiplication $R$-module with finite-dimensional top is pure-injective.
\end{cor}
\begin{proof}
{The proof follows from Corollary \ref{inj}, Theorem \ref{main} and \cite[Theorem 3.5]{Ebrahimipure}.
}
\end{proof}

\section{Conclusion}

This research contributes to the idea of pseudo-absorbing primary multiplication modules over the pullback ring $R = (R_1 \xrightarrow{\mathcal{V}_1(x)} \overline{R} \xleftarrow{\mathcal{V}_2(x)} R_2)$ of two discrete valuation domains $R_1,R_2$ with common factor field $\overline{R}$ and gives a complete description of these modules. The main results of this article include the classification of indecomposable pseudo-absorbing primary multiplication modules over $k[x, y : xy = 0]_{(x,y)}$ (a typical example of a pullback ring of two discrete valuation domains) where $k$ is a field with finite-dimensional top, which is divided into two steps: indecomposable separated pseudo-absorbing primary multiplication and indecomposable non-separated pseudo-absorbing primary multiplication $R$-modules. In the future, this work will be expanded to explore the indecomposable pseudo-absorbing primary multiplication modules with finite-dimensional top over the pullback ring of two discrete valuation domains $R_1,R_2$ onto a semi-simple artinian ring $\overline{R}$ and also over the pullback ring of two local rings $R_1$ and $R_2$ onto a field $\overline{R}$.



\end{document}